\documentclass{amsart}
\usepackage{amsmath}
\usepackage{amsthm}
\usepackage{amssymb}
\usepackage{amsfonts}
\usepackage[applemac]{inputenc}
\usepackage{mathrsfs}
\newcommand{\R}{\mathbb{R}}

\newcommand{\re}{\mathrm{Re}}

\newcommand{\mez}{\frac{1}{2}}  

\renewcommand{\d}{\partial}

\newcommand{\conj}[1]{\overline{#1}}
\newcommand{\modu}[1]{\left|#1\right|}
\newcommand{\pare}[1]{\left(#1\right)}

\newcommand{\wlim}{\rightharpoonup}

\newcommand{\K}{\mathcal K}

\newcommand{\scalar}[2]{\langle#1,#2\rangle}
\def\le{\leqslant}
\def\ge{\geqslant}

\theoremstyle{plain}
\newtheorem{theorem}{Theorem}[section]
\newtheorem{lemma}[theorem]{Lemma}

\newtheorem{proposition}[theorem]{Proposition}

\theoremstyle{definition}

\newtheorem{remark}[theorem]{Remark}
\newtheorem*{remark*}{Remark}

\numberwithin{equation}{section}

\begin{document}
\title[Solitary waves for dipolar quantum gases]
{Existence of solitary waves in dipolar quantum gases}
\author[P. Antonelli]{Paolo Antonelli}
\address[P. Antonelli]{Department of Applied Mathematics and
Theoretical Physics\\
CMS, Wilberforce Road\\ Cambridge CB3 0WA\\ England}
\email{p.antonelli@damtp.cam.ac.uk}
\author[C. Sparber]{Christof Sparber}
\address[C. Sparber]{Department of Mathematics, Statistics, and Computer Science\\
University of Illinois at Chicago\\
851 S. Morgan Street\\
Chicago, IL 60607, USA}
\email{sparber@math.uic.edu}
\begin{abstract}
We study a nonlinear Schr\"odinger equation arising in the 
mean field description of dipolar quantum gases. Under the assumption of sufficiently strong dipolar 
interactions, the existence of standing waves, and hence solitons, is proved together with some of their properties. 
This gives a rigorous argument for the possible existence of solitary waves in Bose-Einstein condensates, which originate 
solely due to the dipolar interaction between the particles.
\end{abstract}

\date{\today}


\thanks{This publication is based on work supported by Award No. KUK-I1-007-43, funded by the King Abdullah University of Science and Technology (KAUST). 
C. Sparber has been supported by the Royal Society through his 
University research fellowship.}

\maketitle

\section{Introduction}\label{intro}

The experimental realization of \emph{Bose-Einstein condensation} (BEC) in dilute gases in 1995 \cite{AEMWC} has marked the beginning of a new era in atomic physics 
and quantum optics. Ever since then, continuous efforts have been undertaken to extend BEC 
physics towards new regimes offering different theoretical and experimental challenges. A particularly interesting 
research field concerns the study of \emph{solitary waves} within BECs, see e.g. \cite{AG} and the references given therein.
More recently, so-called \emph{dipolar BECs}, i.e. condensates made out of particles possessing a permanent electric or magnetic dipole moment \cite{SSZM}, have received much attention. This is due 
to the fact that the additional dipolar interactions between particles are both \emph{long-range} and \emph{non-isotropic} and therefore crucially influence the ground state properties, stability, 
and dynamics of the condensate, see \cite{LMSLP} for a broad review of this subject. In addition, the possibility of a novel class of solitary waves within such systems has been discussed in e.g. \cite{GMS, NPS, PeSa}. 
Motivated by these reports, it is the aim of this paper to rigorously prove the existence of solitary waves within dipolar quantum gases.

To this end, we shall be concerned with the \emph{mean-field description} of (dilute) dipolar quantum gases, based on an 
(augmented) \emph{Gross-Pitaevskii equation}, cf.  \cite{LMSLP} for the validity of such a description. Following \cite{YY, YY1}, we shall describe the time-evolution of such systems 
by the following Gross-Pitaevskii type model
\begin{equation}\label{eq:dip_BEC1}
i\hbar\d_t\psi=-\frac{\hbar^2}{2m}\Delta\psi+g|\psi|^2\psi+\sigma^2(K\ast|\psi|^2)\psi, \quad t,\in \R, x\in \R^3,
\end{equation}
where $g=4\pi\hbar^2Na/m$, for a $N \in \mathbb N$ number of particles, with mass $m>0$ and scattering length $a\in\R$. Finally $\sigma^2\ge0$ denotes the strength of the dipole moment. 
In \eqref{eq:dip_BEC1} we also denote by $‘‘ \ast"$ the convolution w.r.t. $x$ between the local density $\rho=|\psi|^2$ and 
\begin{equation}\label{defn:dipole_pot}
K(x)=\frac{1-3\cos^2\theta}{|x|^3},
\end{equation}
where $\theta=\theta(x)$ is the angle between $x\in\R^3$ and a given (fixed) dipole axis $n\in\R^3$, such that $|n|=1$, i.e.
\begin{equation*}
\cos\theta=\frac{x\cdot n}{|x|}.
\end{equation*}
The non-local potential $K\ast \rho$ describes long-range dipolar interactions between the particles, whereas the local (cubic) nonlinearity $\propto |\psi|^2\psi$ describes the usual 
\emph{contact interaction} between particles (which is short-range, isotropic and characterized by the scattering length $a\in\R$).

For the upcoming mathematical analysis it will be more convenient to rescale \eqref{eq:dip_BEC1} into the following dimensionless form
\begin{equation}\label{eq:dip_BEC}
i\d_t\psi=-\frac{1}{2}\Delta\psi+\lambda_1|\psi|^2\psi+\lambda_2(K\ast|\psi|^2)\psi,
\end{equation}
where $\lambda_1, \lambda_2 \in \R$ are some constants (depending on the physical parameters given above), which describe the strength of the two nonlinearities, respectively. 
The time-dependent equation \eqref{eq:dip_BEC1}, in the presence of an additional (quadratic) confining potential,
has been rigorously analyzed by Carles, Markowich and the second author in \cite{CMS}. Several existence and uniqueness results are discussed in \cite{CMS}, 
as is the possible occurrence of \emph{finite time blow-up} of solutions, which physically corresponding to the possible collapse of the BEC \cite{Lu}, see Remark 1.3 below for more details. 

From the mathematical point of view, it is well known (see e.g. \cite{Caz}) that the possibility of finite time blow-up is strongly linked to the existence of standing wave solutions to \eqref{eq:dip_BEC}, i.e.
\begin{equation*}
\psi(t, x)=e^{i\omega t}u(x),\quad \omega \in \R.
\end{equation*}
Obviously, for $\psi(t,x)$ to be a solution of \eqref{eq:dip_BEC}, the profile $u\in L^2(\R^3)$ has to solve the following nonlinear elliptic equation
\begin{equation}\label{eq:stand_wave}
-\frac{1}{2}\Delta u+\lambda_1|u|^2u+\lambda_2(K\ast|u|^2)u+\omega u=0,
\end{equation}
which will be the main object of study in our work. To this end, we first recall that the total energy associated to associated to \eqref{eq:stand_wave}, is given by
\begin{equation}\label{eq:en}
E(u):=\int_{\R^3}\frac{1}{2}|\nabla u|^2+\frac{\lambda_1}{2}|u|^4+\frac{\lambda_2}{2}(K\ast|u|^2)|u|^2dx
\equiv T(u)+V(u),
\end{equation}
where
\begin{equation}
T(u):= \mez \int_{\R^3}|\nabla u|^2 dx,
\end{equation}
is the kinetic energy, and 
\begin{equation}\label{en_pot}
V(u):=\int_{\R^3}\frac{\lambda_1}{2}|u|^4+\frac{\lambda_2}{2}(K\ast|u|^2)|u|^2 dx, 
\end{equation} 
is the nonlinear potential energy. At this point it might be tempting to study \eqref{eq:stand_wave} via minimization of the energy $E(u)$. However, it is well known, that 
even without the dipole nonlinearity i.e. $\lambda_2 = 0$, this approach fails, since, on the one hand, the energy 
functional in three spatial dimensions 
is found to be \emph{unbounded from below} in the case $\lambda_1<0$ (by invoking the Gagliardo-Nirenberg inequality). On the other hand, when 
$\lambda_2=0$, $\lambda_1>0$ the minimization problem becomes trivial. In other words, equation \eqref{eq:stand_wave} turns out to be \emph{$L^2$-supercritical} in the sense of \cite{Caz} and this 
problem is in fact enhanced by the presence of the dipole nonlinearity. Nevertheless, we shall still follow a variational approach 
for studying the existence of solutions to \eqref{eq:dip_BEC}. This approach is based on the chocie a suitable functional (see Section \ref{sec:existence} below) 
which has first been introduced in \cite{We} in the case of local nonlinearities. Since then, it has proved its use in different applications, 
in particular for water wave models of Davey-Stewartson type \cite{Ci, EE, PSSW}, which have a similar mathematical structure as the model we consider here.

\begin{theorem}\label{thm:main}
Let $\lambda_1, \lambda_2 \in \R$ be such that the following condition holds:
\begin{equation}\label{eq:necess}
\lambda_1<
\left \{
\begin{split} & \frac{4\pi}{3}\lambda_2, \quad  \text{if $\lambda_2 >0$,} \\
  - \, &  \frac{8\pi}{3} \lambda_2,  \quad \text{if $\lambda_2 <0$.}\
\end{split}
\right.
\end{equation}
Then there exists $u\in H^1(\R^3)$ solution to 
\eqref{eq:stand_wave} with corresponding $\omega >0$. 
Furthermore the solution $u$ satisfies the following properties:
\begin{enumerate}
\item $u$ is real-valued and $u(x)\ge 0$, $\forall x\in \R^3$.
\item $u$ is radially symmetric in the $x_1, x_2-$plane and axially symmetric with respect to the $x_3-$axis, i.e.
\begin{gather*}
u(x_1\cos\alpha+x_2\sin\alpha, -x_1\sin\alpha+x_2\cos\alpha, x_3)=u(x_1, x_2, x_3),\\ 
\forall\;\alpha\in[0, 2\pi],\;\textrm{and}\;u(x_1, x_2, -x_3)=u(x_1, x_2, x_3).
\end{gather*}
\item The energy of $u$ satisfies $E(u) = \frac{1}{3}T(u) > 0$. 
\item $u\in H^s(\R^3)$ for all $s\ge 1$.
\item There exist positive constants $C_1, C_2$, such that:
$$
e^{C_1 |x|} (|u(x)| + |\nabla u(x)| ) \le C_2, \ \forall x\in \R^3.
$$
\end{enumerate}
\end{theorem}

The assumptions on $\lambda_1, \lambda_2 \in \R$ can be interpreted as the necessity of a sufficiently strong dipolar nonlinearity. Note that the existence of steady states 
is guaranteed \emph{even in situations where}  $\lambda_1 > 0$, i.e. in the case of a repulsive (defocussing) cubic nonlinearity. The appearance of steady states in this regime is 
therefore \emph{solely} due to the presence of the dipolar interaction and can not be reproduced in a conventional BEC. 
Moreover, by invoking the Galileian-symmetries of \eqref{eq:dip_BEC}, the existence of steady states directly implies the 
existence of \emph{dipolar solitons} in the form
\begin{equation}\label{soliton}
\psi(t,x) = u(x + \kappa t) e^{ i  \omega t} e^{-i \kappa(x+\kappa t)/2 }, \quad \kappa \in \R.
\end{equation}
where $u$ is the smooth, exponentially decaying (as $|x|\to \infty$) non-negative profile guaranteed by Theorem \ref{thm:main}. In the usual language of solitary waves, such a solution 
to \eqref{eq:dip_BEC} is considered to be a \emph{bright soliton} \cite{AG}. Our work therefore provides a rigorous mathematical basis for the existence of dipolar solitons, as studied in \cite{GMS, NPS, PeSa} 
(see also \cite{LMSLP} for a broader discussion and \cite{BAM} for a closely related physical system).

The main drawback of Theorem \ref{thm:main} is that it leaves uniqueness as an open question. In the usual case of a single cubic nonlinearity, uniqueness of (positive) solutions is strongly 
interwoven with the fact that $u=u(|x|)$ is found to be radially symmetric \cite{GMN} (the proof is based on symmetric re-arrangement techniques). 
This \emph{radial symmetry is broken} in our case due to the influence of the dipolar nonlinearity and hence 
we can not conclude uniqueness.

\begin{remark}
In physical experiments one needs to confine the gas by electromagnetic traps in order to achieve sufficiently low temperature. In order to mathematically describe the trapping potential, the 
Gross-Pitaevskii equation usually carries an additional term, namely 
\begin{equation*}
i\d_t\psi=-\frac{1}{2}\Delta\psi+V_{\rm ext} \psi+ \lambda_1|\psi|^2\psi+\lambda_2(K\ast|\psi|^2)\psi,
\end{equation*}
where the $V_{\rm ext}$ is assumed to be of the following form 
$V_{\rm ext}(x)=  \frac{|x|^2}{2} ,$
i.e. a harmonic oscillator confinement. Obviously, the presence of $V_{\rm ext}$ accounts for the existence of steady states even in the linear case $\lambda_1=\lambda_2=0$ (where the Hamiltonian admits 
countable many eigenvalues). 
The situations with trapping potential therefore has to be clearly distinguished from the one considered in Theorem \ref{thm:main} above. In particular, if one assumes the presence of a confinement and in 
addition, say, $\lambda_1\ge \frac{4\pi}{3}\lambda_2 \ge 0$, 
one can easily obtain the existence of stationary states in trapped dipolar BEC by following the arguments given in e.g. \cite{LSY} (see also \cite{BCMW}). That is, 
by minimizing the corresponding energy functional $E(u)$, which is now obviously bounded from below since all terms within $E(u)$ are positive (see also Remark \ref{rem_formula} below). 
Clearly, the presence of $V_{\rm ext}$ breaks the Galileian-symmetry of the model and thus solitary waves of the same kind as given by \eqref{soliton}, 
can only be observed in an actual physical experiment, when the trapping potential is turned off and the BEC is allowed to 
evolve only under the influence of nonlinear effects.
\end{remark}
\begin{remark}
Under Assumption \eqref{eq:necess} on $\lambda_1, \lambda_2$, finite time blow-up (and hence collapse of the condensate \cite{Lu}) can occur for the time-dependent equation \eqref{eq:dip_BEC}. 
Indeed by a straightforward calculation (see also \cite{CMS, Lu}) one obtains the \emph{virial identity}
\begin{align*}
\frac{d^2}{dt^2}I(t)=&\, \int_{\R^3} |\nabla\psi|^2+\frac{3\lambda_1}{2}|\psi|^4
+\frac{3\lambda_2}{2}(K\ast|\psi|^2)|\psi|^2dx\\
=& \, 2E(t)+\frac{1}{2}\int_{\R^3} \lambda_1|\psi|^4+\lambda_2(K\ast|\psi|^2)|\psi|^2dx,
\end{align*}
where
\begin{equation*}
I(t):=\int_{\R^3} \frac{|x|^2}{2}|\psi(t, x)|^2dx
\end{equation*}
Invoking the classical argument of Glassey \cite{G} yields blow-up of solutions to \eqref{eq:dip_BEC} in finite time, provided \eqref{eq:necess} holds true and the \emph{initial energy is negative} (see 
\cite{CMS} for a possible construction of such initial data). 
Note, however, that the solitary wave solutions constructed above correspond to initial data with positive energy (see assertion (3) of Theorem \ref{thm:main}). They do not blow-up in finite time but most probably 
are unstable w.r.t. small perturbations (an issue that can be overcome in experiments by creating effective lower dimensional systems, cf. \cite{BCRL}).
\end{remark}

The paper is now organized as follows: In Section \ref{sec:cond} below we shall as a first step derive necessary conditions for the existence of standing waves. That these waves in 
fact exist is then proved in Section \ref{sec:existence} and we conclude the work by giving the remaining details for the proof of our main theorem in Section \ref{sec:proof}.

\section{Necessary conditions for existence}\label{sec:cond}

In this section we shall prove several conditions which are necessary for the existence of solutions to \eqref{eq:dip_BEC}. Note that without loss of generality, we can impose $n=(0, 0, 1)$. 
In this case $K(x)$ becomes 
\begin{equation}\label{kernel}
K(x)=\frac{x_1^2+x_2^2-2x_3^2}{|x|^5},
\end{equation}
and we shall assume $K$ to be of this form from now on. In the following we also denote the Fourier transform of a function $f(x)$ by
$$
\widehat f(\xi) := \int_{\R^3} f(x) e^{-i \xi \cdot x} d\xi.
$$
We can then recall Lemma 2.1 of \cite{CMS} 
concerning basic properties of the nonlocal potential $K\ast|u|^2$. 

\begin{lemma}\label{lemma:1}
The operator $\K: f\mapsto K\ast f$ can be extended as a continuous operator on $L^p(\R^3)$, for each $1<p<\infty$. Moreover, the Fourier transform of $K$ is given by
\begin{equation}\label{eq:four_trans}
\widehat K(\xi)=\frac{4\pi}{3}(3\cos^2\Theta-1)=\frac{4\pi}{3}\pare{\frac{2\xi_1^2-\xi_2^2-\xi_3^2}{|\xi|^2}},
\end{equation}
where $\Theta= \Theta(\xi)$ denotes the angle between $\xi \in \R^3$ and the dipole axis $n=(0, 0, 1)$.
\end{lemma}
Formula \eqref{eq:four_trans} implies that $\widehat K\in L^\infty(\R^3)$ and thus $K\ast f $ clearly defines a continuous operator $\mathcal K: L^2(\R^3)\to L^2(\R^3)$. In the 
following lemma we shall derive two Pohozaev-type identities, which have to be a-priori satisfied by any solution to \eqref{eq:stand_wave}. 
\begin{lemma}\label{lemma:pohoz}
Let $u\in H^1(\R^3)$ be a solution to \eqref{eq:stand_wave}. Then the following identities hold
\begin{align}
& \int_{\R^3}\frac{1}{2}|\nabla u|^2dx=3\omega\int_{\R^3}|u|^2dx\label{eq:pohoz1}\\
& \int_{\R^3}\frac{\lambda_1}{2}|u|^4+\frac{\lambda_2}{2}(K\ast|u|^2)|u|^2dx=-2\omega\int_{\R^3}|u|^2dx\label{eq:pohoz2}.
\end{align}
Moreover we also have that $V(u)=-\frac{2}{3}T(u)=-2E(u)$.
\end{lemma}
\begin{proof}
We first multiply equation \eqref{eq:stand_wave} by $x\cdot\nabla u$ and then integrate by parts. Straightforward calculations yield
\begin{equation}\label{eq:101}
0=\int_{\R^3}-\frac{1}{4}|\nabla u|^2-\frac{3}{4}\lambda_1|u|^4-\frac{3}{4}\lambda_2(K\ast|u|^2)|u|^2-\frac{3}{2}\omega|u|^2dx.
\end{equation}
On the other hand, multiplying equation \eqref{eq:stand_wave} by $\bar u$, we obtain
\begin{equation}\label{eq:102}
0=\int_{\R^3}\frac{1}{2}|\nabla u|^2+\lambda_1|u|^4+\lambda_2(K\ast|u|^2)|u|^2+\omega|u|^2dx.
\end{equation}
By combining the identities \eqref{eq:101} and \eqref{eq:102}, we infer
\begin{equation*}
0=\int_{\R^3}-\frac{\lambda_1}{2}|u|^4-\frac{\lambda_2}{2}(K\ast|u|^2)|u|^2-2\omega|u|^2dx,
\end{equation*}
which is nothing but \eqref{eq:pohoz2}. In view of \eqref{eq:102}, this also yields \eqref{eq:pohoz1}. Finally, the definitions of the kinetic and potential energy given in Section \ref{intro}, 
together with \eqref{eq:pohoz1} and \eqref{eq:pohoz2} directly imply 
$V(u)=-\frac{2}{3}T(u)=-2E(u).$
\end{proof}

From the Pohozaev-type identities \eqref{eq:pohoz1}, \eqref{eq:pohoz2} we can derive the following necessary conditions for the existence of solutions to \eqref{eq:stand_wave}. 
First of all, \eqref{eq:pohoz1} obviously requires $\omega>0$. Consequently, 
\eqref{eq:pohoz2} implies
\begin{equation}\label{eq:103}
\int_{\R^3}\frac{\lambda_1}{2}|u|^4+\frac{\lambda_2}{2}(K\ast|u|^2)|u|^2dx<0.
\end{equation}
Denoting by $\widehat\rho(\xi)= \widehat {|u|^2}(\xi)$ the Fourier transform of $|u|^2$, Plancherel's theorem allows us to rewrite the left hand side of this inequality in the following form 
\begin{align*}
\int_{\R^3} \lambda_1|u|^4+\lambda_2(K\ast|u|^2)|u|^2dx=  \ \int_{\R^3} \lambda_1\widehat\rho^2(\xi) +\widehat K(\xi) \widehat\rho^2(\xi) d\xi .
\end{align*}
In view of formula \eqref{eq:four_trans} we note that in fact $\widehat K(\xi) \in [-\frac{4\pi} { 3}, \frac{8\pi}{ 3}]$. In the case of $\lambda_2 >0$, this implies
\begin{equation*}
\int_{\R^3} \lambda_1\widehat\rho^2+\lambda_2\widehat K(\xi) \widehat\rho^2(\xi) d\xi \ge\int_{\R^3} \pare{\lambda_1-\frac{4\pi}{3}\lambda_2}\widehat\rho^2(\xi) d\xi,
\end{equation*}
which consequently requires
$
\lambda_1<\frac{4\pi}{3}\lambda_2,
$
as a necessary condition for the existence of solutions, cf. the first line of \eqref{eq:necess}. Similarly, in the case when $\lambda_2 < 0$, we obtain 
\begin{equation*}
\int_{\R^3} \lambda_1\widehat\rho^2+\lambda_2\widehat K(\xi) \widehat\rho^2(\xi) d\xi \ge\int_{\R^3} \pare{\lambda_1+ \frac{8\pi}{3}\lambda_2}\widehat\rho^2(\xi) d\xi,
\end{equation*}
which yields  the second condition given in \eqref{eq:necess}, i.e.
$
\lambda_1<- \frac{8\pi}{3}\lambda_2.
$
\begin{remark} \label{rem_formula} The conditions on $\lambda_1, \lambda_2$ can also be understood as follows: Straightforward calculations show (cf. \cite{LMSLP}) that, for $K(x)$ given by \eqref{kernel}, the 
dipolar nonlinearity can be rewritten as
$$
K\ast |u |^2 =  - \frac{4 \pi}{3} | u |^2 - \frac{\partial^2 }{\partial x_3^2}\Phi  ,
$$
where $\Phi$ solves the Poisson equation $- \Delta \Phi =  | u |^2$. The total nonlinear potential energy \eqref{en_pot} is therefore given by 
\begin{equation}\label{V}
V(u)= \frac{1}{2}  \int_{\R^3} (\lambda_1-  \frac{4 \pi}{3} \lambda_2) |u|^4 dx +\frac{\lambda_2}{2}\int_{\R^3} |\partial_{x_3} \nabla \Phi |^2 dx,
\end{equation}
i.e. a term which stems from a (combined) cubic nonlinearity and a sub-critical term, which stems from Poisson's equation. The assumptions on $\lambda_1, \lambda_2$ consequently 
ensure that the total potential energy is essentially the same as in the case of an attractive cubic nonlinearity. 
Note that \eqref{V} implies that in the \emph{defocusing situation}, i.e. $\lambda_1 \ge \frac{4\pi}{3} \lambda_2 \ge 0$, the nonlinear potential energy is indeed convex (see also \cite{BCMW}).
\end{remark}

\section{A variational formulation}\label{sec:existence}

We shall now formulate a variational problem, which will be used to ensure the existence of solutions to \eqref{eq:stand_wave}. 
To this end, we introduce 
\begin{equation}\label{eq:funct}
J(v):=
\frac{\|\nabla v \|_{L^2}^3\|v \|_{L^2}}{-\lambda_1\|v \|_{L^4}^4-\lambda_2\scalar{\K(|v |^2)}{|v |^2}},
\end{equation}
where $\langle \cdot, \cdot \rangle $ denotes the scalar product in $L^2(\R^3)$.
This functional is well-defined for each $v \in H^1(\R^3)$ in view of Lemma \ref{lemma:1}. 
Moreover, $J(v)$ satisfies the following scaling properties: Let
\begin{equation}\label{eq:scal}
v_{q, s}(x):=q v (sx), \quad q,s>0,
\end{equation}
which implies
\begin{equation*}
\| v_{q, s}\|_{L^2}^2=q^2s^{-3}\|v \|_{L^2},\;\|\nabla v_{q, s}\|_{L^2}^2=q^2s^{-1}\|\nabla v\|_{L^2}, \;
\| v_{q, s}\|_{L^4}^4=q^4s^{-3}\|v \|_{L^4}^4.
\end{equation*}
Then $J(v)$ is found to be invariant under this scaling above, i.e. $J(v_{s, q})=J(v)$. To this end, one checks that the nonlocal term $\int(K\ast|v|^2)|v|^2dx$ scales in the same way as 
the $L^4(\R^3)$ norm, since
\begin{align*}
\int_{\R^3} (K\ast| v_{q, s}|^2)|v_{q, s}|^2\, dx= & \ q^4s^{-6}\int_{\R^3} \widehat K(\xi)\widehat\rho^2\left(\frac{\xi}{s} \right)d\xi\\
= & \ q^4s^{-3}\int_{\R^3} \widehat K(s\xi)\widehat\rho^2(\xi)\, d\xi=q^4s^{-3}\int_{\R^3} (K\ast|v|^2)|v|^2\, dx,
\end{align*}
where $\widehat\rho$ denotes the Fourier transform of $\rho:=|v|^2$. Note that here we have used the fact that $\widehat K(s\xi)=\widehat K(\xi)$, for each 
$s>0$, which is easily seen from \eqref{eq:four_trans}.
\begin{lemma}\label{lemma:4}
Let $v\in H^1(\R^3)$ be a critical point of the functional $J$, with $J(v)= \alpha \in \R$. Then $v$ is a weak solution to
\begin{equation}
- \beta \Delta v + 4\alpha(\lambda_1|v|^2v+\lambda_2\K(|v|^2)v)+ \widetilde \omega v=0,
\end{equation}
where $\widetilde \omega = \|\nabla v\|_{L^2}^3\|v\|_{L^2}^{-1} $ and $\beta = 3 \|v\|_{L^2} \|\nabla v\|_{L^2}$.
\end{lemma}
\begin{proof}
Denote $\beta_1 = \| v\|_{L^2}$ and $\beta_2 = \|\nabla v\|_{L^2}$. All we need to do is to find the points $v\in H^1(\R^3)$, where the first variation of $J$ satisfies $\delta J(v)=0$. 
To this end, we calculate the variation of the nonlocal 
term $\scalar{\K(|v|^2)}{|v|^2}$ in the form 
$\scalar{2v\K'(|v|^2)}{\eta}$, $\forall \, \eta\in C_0^\infty(\R^3)$, where $\K'(f)$ denotes the Fr\'echet derivative of the operator 
$\K(f)$. In order to compute $\K'(f)$ we use Plancherel's theorem, to write
\begin{equation*}
\scalar{\K(f+\eta)}{f+\eta}-\scalar{\K(f)}{f}=\int_{\R^3} \widehat K\pare{2\re(f\conj{\eta})+|\eta|^2}dx. 
\end{equation*}
Using \eqref{eq:four_trans}, we obtain
\begin{equation*}
\modu{\scalar{\K(f+\eta)}{f-\eta}-\scalar{\K(f)}{f}-\scalar{2\K(f)}{\eta}}\le\frac{8\pi}{3}\|\eta\|_{L^2}^2
\end{equation*}
and we consequently infer that the Fr\'echet derivative of $K(f)$ is given by $\scalar{\K'(f)}{\eta}=\scalar{2\K(f)}{\eta}$. We therefore find that $\delta J(v)$ is given by
\begin{align*}
\delta J(v)=& \ -\frac{3\|v\|_{L^2}\|\nabla v\|_{L^2}\scalar{-\Delta v}{\eta}
+\|\nabla v\|_{L^2}^3\|v\|_{L^2}^{-1}\scalar{v}{\eta}}{(\lambda_1\|v\|_{L^4}^4+\lambda_2\scalar{\K(|v|^2)}{|v|^2})}\\
&+\frac{\|\nabla v\|_{L^2}^3\|v\|_{L^2}\pare{4\lambda_1\scalar{|v|^2v}{\eta}+4\lambda_2\scalar{\K(|v|^2)v}{\eta}}}
{\pare{\lambda_1\|v\|_{L^4}^4+\lambda_2\scalar{\K(|v|^2)}{|v|^2}}^2}\\
=& \ \frac{1}{(\lambda_1\|v\|_{L^4}^4+\lambda_2\scalar{\K(|v|^2)}{|v|^2})}
\Big(3\beta_1\beta_2\scalar{\Delta v}{\eta}-\beta_1^{-1}\beta_2^3\scalar{v}{\eta}\\
& \ -4\alpha\pare{\lambda_1\scalar{v}{\eta}+\lambda_2\scalar{\K(|v|^2v)}{\eta}}\Big).
\end{align*}
Since this has to hold for any $\eta\in C_0^\infty(\R^3)$, we conclude
\begin{equation*}
3\beta_1\beta_2\Delta v-\beta_1^{-1}\beta_2^3v-4\alpha\pare{\lambda_1|v|^2v+\lambda_2\K(|v|^2v)}=0,
\end{equation*}
which proves the assertion of the lemma.
\end{proof}

Consequently the problem of existence of solutions to \eqref{eq:stand_wave} reduces to the task of finding critical points, say minima, of the functional $J(v)$.
\begin{proposition}\label{prop:exist} Let $\lambda_1, \lambda_2$ satisfy assumption \eqref{eq:necess}. 
Then there exists a minimizer $v_*\in H^1(\R^3)$ for the functional defined in \eqref{eq:funct}, i.e.
\begin{equation}
J(v_*)=j:=\inf_{0\neq v\in H^1(\R^3)}J(v).
\end{equation}
\end{proposition}

\begin{proof} First we note that condition \eqref{eq:necess} ensures that $J(v)$ is non-negative. Hence, 
there exists a minimizing sequence  $\{v_n\}_{n\in \mathbb N}\subset H^1(\R^3)$, i.e.
\begin{equation*}
j=\lim_{n\to\infty}J(v_n).
\end{equation*}
Furthermore, we can rescale the sequence $\{v_n\}_{n\in \mathbb N}$, such that 
\begin{equation*}
\|v_n\|_{L^2}=1,\qquad\|\nabla v_n\|_{L^2}=1
\end{equation*}
and thus
\begin{equation}\label{eq:104}
j=\lim_{n\to\infty}J(v_n)=-\left( \lim_{n\to\infty}V(v_n)\right)^{-1}
\end{equation}
Since $\{v_n\}_{n\in \mathbb N}$ is uniformly bounded in $H^1(\R^3)$, we know that, up to extraction of a sub-sequence,  there exists a weak limit in $H^1(\R^3)$, i.e. 
\begin{equation*}
v_n\wlim v_* \quad\textrm{in}\;H^1(\R^3).
\end{equation*}
Moreover, by the lower semicontinuity of the $L^2$ norm, we have
\begin{equation*}
\|v_* \|_{L^2}\equiv  b_1\le1,\qquad\|\nabla v_* \|_{L^2}\equiv b_2\le1.
\end{equation*}
First of all, by arguing as in \cite{PSSW} we can show that both $b_1$ and $b_2$ are strictly positive: Indeed, from the boundedness of $J(v_n)<\infty$, the fact that 
$\|v_n\|_{L^2}=\|\nabla v_n\|_{L^2}=1$, and assumption \eqref{eq:necess}, we infer that the 
$L^4-$norm of the sequence $(v_n)_{n \in \mathbb N}$ is uniformly bounded away from zero, i.e. $\|v_n\|_{L^4}\ge C>0$. Hence by arguing as in Section 5 of \cite{PSSW} (which in itself is based on a result by Lieb \cite{L}), 
we conclude that $b_1\neq0$ and $b_2\neq0$.

Thus if we can prove that  $\{v_n\}_{n\in \mathbb N}$ indeed converges (not only weakly but) strongly towards $v_* \in H^1(\R^3)$, then clearly $v_* $ is a minimizer for the functional \eqref{eq:funct}. To this end, it suffices to prove 
that indeed $b_1= b_2=1$.

We consequently define the difference $w_n:=v_n-v_*$, for which we have
$$\|w_n\|_{L^2}^2 \stackrel{n \rightarrow \infty}{\longrightarrow} 1-b_1^2, \quad 
\|\nabla w_n\|_{L^2}^2 \stackrel{n \rightarrow \infty}{\longrightarrow}  1-b_2^2.$$
Hence
\begin{equation*}
j\le\lim_{n\to\infty}J(w_n)=\frac{(1-b_2^2)^{3/2}(1-b_1^2)^{1/2}}{-\lim_{n\to\infty}2V(w_n)}.
\end{equation*}
We now want to study the limit of $2V(w_n)$, as $n\to \infty$, which in view of \eqref{eq:104} can be written as
\begin{equation*}
\lim_{n\to\infty}2V(w_n)
=\lim_{n\to\infty}\pare{\lambda_1(\|w_n\|_{L^4}^4-\|v_n\|_{L^4}^4)+\lambda_2(\scalar{\K(w_n^2)}{w_n^2}
-\scalar{\K(v_n^2)}{v_n^2})}
-\frac{1}{j},
\end{equation*}
We first consider the term $\|w_n\|_{L^4}^4-\|v_n\|_{L^4}^4$. Using a classical result by Brezis and Lieb \cite{BL} we know that 
\begin{equation*}
\lim_{n\to\infty}\pare{\|w_n\|_{L^4}^4-\|v_n\|_{L^4}^4}=-\|v_*\|_{L^4}^4.
\end{equation*}
In a second step we rewrite the term $\scalar{\K(w_n^2)}{w_n^2}$ in the following way
\begin{align*}
\scalar{\K(w_n^2)}{w_n^2}=& \ \int_{\R^3}  w_n^2\K(v_n^2)dx+\int_{\R^3}  w_n^2\K(v_*^2)dx
-2\int_{\R^3}  v_nv_*\K(w_n^2)dx\\
=& \ \scalar{\K(v_n^2)}{v_n^2}+\int_{\R^3}  w_n^2\K(v_*^2)dx+\int_{\R^3} (v_*^2-2v_*v_n)\K(v_n^2)dx\\
& \ -2\int_{\R^3}  v_nv_*\K(w_n^2)dx
\end{align*}
By Sobolev embeddings, the weak convergence $w_n\wlim0$ in $H^1(\R^3)$ implies
$w_n^2\wlim0$ in $L^p(\R^3)$, for $1\le p\le3$. Hence we immediately get
\begin{equation*}
\int_{\R^3}  w_n^2\K(v_*^2)dx \stackrel{n \rightarrow \infty}{\longrightarrow}  0,
\end{equation*}
since $\|\K(v_*^2)\|_{L^2}^2\le\frac{8\pi}{3}\|v_*^2\|_{L^2}$. 
Next, we want to show that
\begin{equation}\label{eq:des}
\int_{\R^3} (v_*^2-2v_*v_n)\K(v_n^2)dx  \stackrel{n \rightarrow \infty}{\longrightarrow}  -\int_{\R^3}  v_*^2\K(v_*^2)dx.
\end{equation}
Indeed,
\begin{align*}
 \modu{\int_{\R^3} (v_*^2-2v_nv_*)\K(v_n^2)dx+\int_{\R^3}  v_*^2\K(v_*^2)dx} 
 \le & \ \modu{\int _{\R^3} w_n(v_n+v_*)\K(v_*^2)dx}\\
& \ +2\modu{\int_{\R^3}  w_nv_*\K(v_n^2)dx},
\end{align*}
and thus, by the Cauchy-Schwarz inequality
\begin{align*}
  & \modu{\int_{\R^3} (v_*^2-2v_nv_*)\K(v_n^2)dx+\int_{\R^3}  v_*^2\K(v_*^2)dx}   
  \le  \pare{\int_{\R^3}  w_n^2\K^2(v_*^2)dx}^{1/2}\times \\
  & \times \pare{\int_{\R^3} (v_n+v_*)^2dx}^{1/2}  +2\pare{\int_{\R^3}  w_n^2v_*^2dx}^{1/2}\pare{\int_{\R^3} \K^2(v_n^2)dx}^{1/2}.
\end{align*}
Having in mind, that $\mathcal \K$ (cf. Lemma \ref{lemma:1}) is bounded in $L^4(\R^3)$, we can again use the convergence of $w_n^2\wlim0$ in $L^2(\R^3)$, to obtain \eqref{eq:des}. 
By similar arguments one can show that the term $-2\int v_nv_*\K(w_n^2)dx$ converges to zero.
\smallskip

In summary, we conclude that 
\begin{equation*}
-\lim_{n\to\infty}2V(w_n)=2V(v_*)+\frac{1}{j},
\end{equation*}
which consequently implies
\begin{equation*}
j\le J(w_n)\stackrel{n \rightarrow \infty}{\longrightarrow} \frac{(1-b_1^2)^{1/2}(1-b_2^{2/3})^{3/2}}{2V(v_*)+\frac{1}{j}}.
\end{equation*}
Rearranging this inequality and using the fact that $-j 2V(v_*)=b_1b_2^3$ we obtain
\begin{equation}\label{eq:ineq}
\frac{1}{j}\le\frac{(1-b_1^2)^{1/2}(1-b_2^{2/3})^{3/2}}{j}+\frac{b_1b_2^3}{j}.
\end{equation}
On the other hand, by simple algebra we obtain
$$(1-b_1^2)^{1/2}(1-b_2^{2/3})^{3/2}+b_1b_2^3\le1, $$ 
and thus, also the reverse inequality corresponding to \eqref{eq:ineq} holds. Thus
\begin{equation*}
\frac{1}{j}=\frac{(1-b_1^2)^{1/2}(1-b_2^{2/3})^{3/2}}{j}+\frac{b_1b_2^3}{j}.
\end{equation*}
This identity holds if and only if, either: $b_1=b_2=0$, which is not allowed in our study, or else, if $b_1=b_2=1$. We therefore conclude that 
\begin{equation*} 
\|v_*\|_{L^2}=1, \qquad\|\nabla v_*\|_{L^2}=1,
\end{equation*}
which consequently implies the strong convergence of the minimizing sequence. In summary, this shows the infimum of the functional is attained at $v_*$. 
\end{proof}

\section{Proof of Theorem \ref{thm:main}}\label{sec:proof}

Recall that the functional $J(v)$ is invariant under the scaling \eqref{eq:scal}. Thus we can choose the minimizer 
$v_*$ to be such that $\|v_*\|_{L^2}=\beta_1, \|\nabla v_*\|_{L^2}=\beta_2$. Hence, by combining the results stated in 
Proposition \ref{prop:exist} and Lemma \ref{lemma:4}, we conclude that $v_*$ is a solution to 
\begin{equation}\label{pre}
- 3\beta_1\beta_2\Delta v_*+ 4j\pare{\lambda_1|v_*|^2v_*+\lambda_2\K(|v_*|^2)v_*}+ \widetilde \omega v_*  =0,
\end{equation}
with $\widetilde \omega =  \beta_1^{-1}\beta_2^3$.
Since $\beta_1, \beta_2>0$ are arbitrary, we can choose these parameters in the following way: 
$\beta_1=\frac{1}{6}\pare{\frac{\omega}{6}}^{-1/4}$, 
$\beta_2=\pare{\frac{\omega}{6}}^{1/4}$, so that \eqref{pre} becomes
\begin{equation}
- \frac{1}{2}\Delta v_*+ 4j\pare{\lambda_1|v_*|^2v_*+\lambda_2\K(|v_*|^2)v_*} + \omega v_* =0,
\end{equation}
where $j\equiv J(v_*)$, as in Proposition \ref{prop:exist}. Hence, by rescaling $u(x) = (4j)^{1/2}v_*(x)$, the existence of a weak solution $u\in H^1(\R^3)$ to the original problem \eqref{eq:stand_wave} is proved.

It remains to prove Assertion (1) - (5) of Theorem \ref{thm:main}:  In order to show Assertion (1) we note that, by straightforward calculations
\begin{align*}
\| \nabla |v| \|_{L^2} \le \| \nabla v \|_{L^2}, \quad \forall v \in H^1(\R^3).
\end{align*}
Thus $J(|v|) \le J(v)$ and hence $J(v_*) = J(|v_*|)$, which implies that the minimizer satisfies $v_*(x) \ge 0$. To prove Assertion (2) we can use Steiner symmetrization (cf. \cite{S}, proof of Theorem 3, and \cite{ATDL}. See also \cite{LL} for 
the  a closely related topic of re-arrangement inequalities.). Let $u^\#$ be the Steiner symmetrization of $u$ around a plane in $\R^3$. It is straightforward, cf. \cite[Theorem 3] {S} to prove that $J(u^\#)\le J(u)$, provided that $K^\#=K$, where $K^\#$ is the Steiner symmetrization of $K$. This means that we can find a minimizer $u$ of $J$ which has the same symmetries of $K$, hence it is radially symmetric in the $(x_1, x_2)-$plane and axially symmetric with respect to the $x_3-$axis. 
Assertion (3) then directly follows 
from the identity 
$$E(u) = \frac{1}{3} T(u)>0,$$ 
as proved in Lemma \ref{lemma:pohoz}. Assertion (4) follows from the simple observation that for $u \in H^1(\R^3)$ the right hand side of 
\begin{equation*}
-\frac{1}{2}\Delta u= - \lambda_1|u|^2u - \lambda_2(K\ast|u|^2)u-\omega u,
\end{equation*}
is bounded in $L^2(\R^3)$, due to the properties of $\mathcal K$ stated in Lemma \ref{lemma:1} and the Gagliardo-Nirenberg inequality $\|u\|_{L^4}^4\le \|u\|_{L^2} \|\nabla u\|_{L^2}^3 $. 
Thus we conclude that in fact $u\in H^2(\R^3)$ and an induction argument shows $u\in H^s(\R^3)$ for all $s\ge 1$.

Finally, in order to prove the asymptotic decay as $|x|\to \infty$, all one has to do is follow the arguments given in Step 6 of the proof of \cite[Theorem 2.4]{Ci} (which in itself 
follows from \cite{Caz}). 

\begin{remark}
As a by-product of our analysis we obtain, that there exists a $C>0$, such that for any $f\in H^1(\R^3)$ the following inequality holds
\begin{equation*}
-\lambda_1\|f\|_{L^4}^4-\lambda_2\scalar{K\ast|f|^2}{|f|^2}\le C\|\nabla f\|_{L^2}^3\|f\|_{L^2},
\end{equation*}
where the optimal constant $C=C_*$ is given by 
$$C_{*}=\frac{-\lambda_1\|v_*\|_{L^4}^4-\lambda_2\scalar{K\ast|v_*|^2}{|v_*|^2}}{\|\nabla v_*\|_{L^2}^3\|v_*\|_{L^2}},$$
with $v_*$ being the minimizer of $J(v)$, as guaranteed by Proposition \ref{prop:exist}.
\end{remark}

\end{document}